\documentclass[letter, 10pt, conference]{ieeeconf}

\IEEEoverridecommandlockouts                              

\usepackage{graphicx}
\usepackage{amsmath}
\usepackage{amssymb}

\usepackage{dsfont}
\usepackage{tikz}

\usepackage{amsthm}
\newtheorem{theorem}{Theorem}
\newtheorem{property}[theorem]{Property}
\newtheorem{lemma}[theorem]{Lemma}
\newtheorem{proposition}[theorem]{Proposition}

\newtheorem{remark}{Remark}

\usepackage{xcolor}

\definecolor{cof}{RGB}{219,144,71}
\definecolor{pur}{RGB}{186,146,162}
\definecolor{greeo}{RGB}{91,173,69}
\definecolor{greet}{RGB}{52,111,72}

\newcommand{\bal}[1] {\ensuremath{\left(\begin{array}{#1}}}
\newcommand{\ear} {\ensuremath{\end{array}\right)}}

\newcommand{\bdiag}[1] {\ensuremath{\mathrm{block.diag}(#1)}} 
\newcommand{\bals}[1] {\ensuremath{\left[\begin{array}{#1}}} 
\newcommand{\ears} {\ensuremath{\end{array} \right] }} 

\DeclareMathOperator{\rank}{rank}

\DeclareMathOperator{\trace}{trace}
\DeclareMathOperator{\sign}{sign}

\DeclareMathOperator{\SO}{SO}
\DeclareMathOperator{\sinc}{sinc}

\newcommand{\T}{\ensuremath{\top}}

\newcommand{\funcRdR}{\ensuremath{{f}}}
\newcommand{\funcRdRd}{\ensuremath{{X}}}

\DeclareMathOperator*{\bigtimes}{\raisebox{-0.3ex}{\text{\Large$\times$}}}

\let\leq\leqslant
\let\geq\geqslant
\let\emptyset\varnothing

\newcommand{\calC}{\ensuremath{\mathcal{C}}}

\newcommand{\calE}{\ensuremath{\mathcal{E}}}
\newcommand{\calF}{\ensuremath{\mathcal{F}}}
\newcommand{\calG}{\ensuremath{\mathcal{G}}}

\newcommand{\calI}{\ensuremath{\mathcal{I}}}

\newcommand{\calL}{\ensuremath{\mathcal{L}}}

\newcommand{\calR}{\ensuremath{\mathcal{R}}}

\newcommand{\calX}{\ensuremath{\mathcal{X}}}

\newcommand{\frakso}{\ensuremath{\mathfrak{so}}}

\newcommand{\bmat}{\begin{matrix}}
\newcommand{\emat}{\end{matrix}}
\newcommand{\bbm}{\begin{bmatrix}}
\newcommand{\ebm}{\end{bmatrix}}
\newcommand{\bpm}{\begin{pmatrix}}
\newcommand{\epm}{\end{pmatrix}}
\newcommand{\bse}{\begin{subequations}}
\newcommand{\ese}{\end{subequations}}
\newcommand{\beq}{\begin{equation}}
\newcommand{\eeq}{\end{equation}}
\newcommand{\ben}{\begin{enumerate}}
\newcommand{\een}{\end{enumerate}}
\newcommand{\beni}{\renewcommand{\labelenumi}{\roman{enumi}.}
\renewcommand{\theenumi}{\roman{enumi}}\begin{enumerate}}
\newcommand{\eeni}{\end{enumerate}\renewcommand{\labelenumi}{\arabic{enumi}.}
\renewcommand{\theenumi}{\arabic{enumi}}}
\newcommand{\bena}{\renewcommand{\labelenumi}{\alpha{enumi}.}
\renewcommand{\theenumi}{\alpha{enumi}}\begin{enumerate}}
\newcommand{\eena}{\end{enumerate}\renewcommand{\labelenumi}{\arabic{enumi}.}
\renewcommand{\theenumi}{\arabic{enumi}}}
\newcommand{\bit}{\begin{itemize}}
\newcommand{\eit}{\end{itemize}}

\newcommand{\R}{\ensuremath{\mathbb R}}

\title{\LARGE \bf
Finite-time attitude synchronization with a discontinuous protocol}

\author{Jieqiang Wei, Silun Zhang, Antonio Adaldo, Xiaoming Hu and Karl H. Johansson 
 \thanks{*This work is supported by Knut and Alice Wallenberg Foundation, Swedish Research Council, and Swedish Foundation for Strategic Research.}
 \thanks{J. Wei, A. Adaldo and K.H. Johansson are with the ACCESS Linnaeus Centre, School of Electrical Engineering. 
 S. Zhang and X. Hu are with School of engineering sciences. 
 KTH Royal Institute of Technology,
 SE-100 44 Stockholm, Sweden. Emails:
         {\tt\small \{jieqiang, silunz, adaldo, hu, kallej\}@kth.se}}
}

\begin{document}

\maketitle
\thispagestyle{plain}
\pagestyle{plain}

\begin{abstract}\label{s:Abstract}
A finite-time attitude synchronization problem is considered in this paper where the rotation of each rigid body is expressed using the axis-angle representation. One simple discontinuous and distributed controller using the vectorized signum function is proposed. This controller only involves the sign of the state differences of adjacent neighbors. In order to avoid the singularity introduced by the axis-angular representation, an extra constraint is added to the initial condition. It is proved that for some initial conditions, the control law achieves finite-time attitude synchronization. One simulated example is provided to verify the usage of the control protocol designed in this paper. 
\end{abstract}

\section{Introduction}\label{s:Introduction}

Motivated by aerospace developments in the middle of the last century~\cite{Bower1964,Kowalik1970}, the rigid-body attitude control has attracted considerable attentions with many promising applications such as aircraft attitude control \cite{Athanasopoulos2014,Tsiotras1994}, spacial grabbing technology of manipulators~\cite{ZXLi}, target surveillance by unmanned vehicles~\cite{pettersen1996position}, camera calibration in computer vision~\cite{ma2012invitation}. Furthermore, the configuration space of rigid-body attitudes is a compact non-Euclidean manifold $SO(3)$, which poses more theoretical challenges for the attitude system control \cite{Bhat00scl}.

Following many notable results on the control for a single attitude, in last decades the coordination of multiple attitudes has been of high interest.
Based on a passivity approach, \cite{ren2010distributed} proposed a consensus control protocol for multiple rigid bodies with attitudes represented by modified Rodrigues parameters. As the attitude system evolves in $SO(3)$ a compact manifold without a boundary, there exist no continuous control law to achieve globally asymptotic stability in closed-loop system. In \cite{Thunberg2014auto}, a proposed methodology based on axis-angle representation obtains almost global asymptotic consensus for attitude synchronization. Besides these agreement results, \cite{lee2012relative,WJ15ac} provided distributed control for the cooperative formation in attitude space.

Among all the studies about attitude synchronization, finite time convergence problem is an important topic and has been mainly studied using continuous control protocols, see e.g., \cite{Du2011,Zong2016}. In this paper we shall focus on the finite time attitude synchronization problem using discontinuous control laws. The discontinuous strategy is motivated by the success of binary controller using signum function in the scalar multi-agent systems, see e.g. \cite{Chen2011,LiuLam2016,Cortes2006,Hui2010}. Nonsmooth analysis is employed to prove the finite-time synchronization rigorously.

The structure of the paper is as follows. In Section \ref{s:Preliminaries}, we introduce some terminologies and notations in the context of  graph theory and stability analysis of discontinuous dynamical systems. Section \ref{ss:basic_model} presents the problem formulation of finite time consensus problem. The main result of the stability analysis of the finite time convergence are presented in Section~\ref{s:vector}. In Section \ref{s:simulation}, one example is demonstrated to verify the main result and to show the limitation of it. 
Then the conclusion follows. 

The notation used in this paper is collected here.

\textbf{Notation}. With $\R_-$ and $\R_{\geqslant 0}$ we denote the sets of negative and nonnegative real numbers respectively. The $i$th row of a matrix $M$ is denoted as $M_{i}$. 
For any matrix $M$, we denote $M\otimes I$ as $\hat{M}$ and $M_{i}\otimes I$ as $\hat{M}_i$. The vectors $e_1,e_2,\ldots,e_n$ denote the canonical basis of $\R^n$. The set $SO(3)=\{R\in\R^{3\times 3}: RR^\top = I, \det R= 1 \}.$ The vector space of real $n$ by $n$ skew symmetric matrices is denoted as $\frakso(3)$.  The vector $\mathds{1}$ denotes a column vector with all components equal to one. For any number $a\in\R$, the sign function is defined as 
\begin{equation}\label{e:scalar_sign}
\sign(a) = \begin{cases}
1 & \textrm{ if } a>0,\\
0 & \textrm{ if } a=0,\\
-1 & \textrm{ if } a<0.
\end{cases}
\end{equation}
For vectors, the signum function is defined component-wisely in this paper. $\|\cdot\|_p$ denotes the $\ell_p$-norm and  the $\ell_2$-norm is sometimes denoted simply as $\|\cdot\|$ without subscript.


\section{Preliminaries}\label{s:Preliminaries}

In this section, we briefly review some essentials from rigid body attitude,  graph theory, as can be found in, e.g., \cite{biggs1993algebraic,Bollobas98}, and give some definitions and notations regarding Filippov solutions.

For any real matrix $A\in\R^{n\times n}$, its exponential $e^{A}$ is a well-defined matrix. 
\begin{lemma}
The exponential map 
\begin{equation}
\exp :\frakso(3)\rightarrow SO(3)
\end{equation}
is surjective. 
\end{lemma}

For any $p\in\R^3$ and $\hat{p}$ given as
\begin{equation}\label{e:basic:hat}
 \hat{p}:=
 \begin{pmatrix}
   0 &-p_3 &p_2\\
   p_3& 0 &-p_1\\
   -p_2& p_1 & 0
 \end{pmatrix},
\end{equation}
the Rodrigues' formula shows that 
\begin{equation}\label{e:exp-map}
e^{\hat{p}} = I_3+\frac{\sin(\theta)}{\theta}\hat{p}+\frac{1-\cos(\theta)}{\theta^2}(\hat{p})^2
\end{equation}
where $\theta=\|p\|_2$. In other words, $e^{\hat{p}}$ is the rotation matrix through an angle $\theta$ anticlockwise about the axis spanned by $p$. For $R\in SO(3)$ with $R\neq I_3$, the inverse of exponential map \eqref{e:exp-map} is given as 
\begin{equation}\label{e:log-map}
\log(R) = \frac{\theta}{2\sin(\theta)}(R-R^\top)
\end{equation}
where $\theta=\arccos(\frac{\trace(R)-1}{2})$. The Riemannian metric in $SO(3)$ is defined as $d_R(R_1, R_2)=\frac{1}{\sqrt{2}}\|\log(R^{-1}_1R_2)\|_F$ where $\|\cdot\|_F$ is the Frobenius norm for the matrix.  

\begin{proposition}[Euler]
Any orientation $R\in SO(3)$ is equivalent to a rotation about a fixed axis $\omega\in\R^3$ through an angle $\theta\in[-\pi, \pi)$. 
\end{proposition}

Based on the previous proposition, we have that the open ball $B_\pi(I)$ in $SO(3)$ with radius $\pi$ around the identity is almost the whole $SO(3)$. Furthermore, the open ball $B_\pi(I)$ is diffeomorphic to the open ball $B_{\pi}(0)$ in $\R^3$ via \eqref{e:basic:hat}, logarithmic and exponential map defined in \eqref{e:log-map} and \eqref{e:exp-map}, respectively. We call the representation of a matrix in $SO(3)$ in $\R^3$ as its axis-angle representation.

\medskip

An undirected \emph{graph} $\calG=(\calI,\calE)$ consists of a finite set of nodes $\calI = \{1,2,\ldots,n\}$ and a set of edges $\calE\subset\calI\times\calI$ of unordered pairs of $\calI$. To any edge $(i,j)\in\calE$, we associate a weight $w_{ij}>0$. 
Next, we say that a graph $\calG$ is connected if, for any two nodes $i$ and $j$, there exists a sequence of edges that connects them. If the edges are ordered pairs of $\calI$, the graph $\calG$ is called a \emph{directed graph}, or \emph{digraph} for short.
An edge of a digraph $\mathcal{G}$ is denoted by $(i,j)$ (with $i\neq j$) representing the tail vertex $i$ and the head
vertex $j$ of this edge. A digraph is completely specified by its \emph{incidence
matrix} $B\in\R^{n\times m}$, where $|\calE|=m$, with $B_{ij}$
equal to $-1$ if the $j$th edge is towards vertex
$i$, and equal to $1$ if the $j$th edge is originating from
vertex $i$, and $0$ otherwise. The incidence matrix for undirected graphs is defined by adding arbitrary orientations to the edges of the graph.


\medskip

In the remainder of this section we give definitions and notation regarding Filippov solutions (see, e.g., \cite{filippov1988,cortes2008}) that will be used in this paper. Let $\funcRdRd$ be a map from $\R^n$ to $\R^n$ and let $2^{\R^n}$ denote the collection of all subsets of $\R^n$. Then, the \emph{Filippov set-valued map} of $\funcRdRd$, denoted $\calF[\funcRdRd]:\R^n\rightarrow 2^{\R^n}$, is defined as
\begin{equation}
\calF[\funcRdRd](x) \triangleq \bigcap_{\delta>0}\bigcap_{\mu(S)=0}\overline{\mathrm{co}}\big\{ \funcRdRd(B(x,\delta)\backslash S) \big\},
\label{eqn_Filippovdef}
\end{equation}
where $B(x,\delta)$ is the open ball centered at $x$ with radius $\delta>0$, $S$ is a subset of $\R^n$, $\mu$ denotes the Lebesgue measure and $\overline{\mathrm{co}}\{\calX\}$ denotes the convex closure of a set $\calX$. If $X$ is continuous at $x$, then $ \calF[\funcRdRd](x)$ contains only the point $X(x)$.

\begin{property}[Calculus for $\calF$, \cite{paden1987}]\label{p:calculus for Filippov}
The following properties hold for the Filippov set-valued map (\ref{eqn_Filippovdef}):
	\begin{enumerate}
		\item Assume that $f:\R^m\rightarrow\R^n$ is locally bounded. Then $\exists N_f\subset \R^m$, $\mu(N_f)=0$ such that $\forall N\subset\R^m, \mu(N)=0$,
		\begin{equation}
		\calF[f](x)=\mathrm{co}\{\lim_{i\rightarrow\infty} f(x_i)\mid x_i\rightarrow x, x_i\notin N_f\cup N \}.
		\end{equation}
		\item Assume that $f_j:\R^m\rightarrow \R^{n_j}$, $j=1,\ldots,N$ are locally bounded, then
		\begin{equation}
		\calF\bigg[ \bigtimes_{j=1}^N f_j \bigg](x) \subset \bigtimes_{j=1}^N\calF[f_j](x).
		\end{equation}
		\item Let $g:\R^m\rightarrow\R^n$ be $C^1$, $\rank Dg(x)=n$ and $f:\R^n\rightarrow\R^p$ be locally bounded; then
		\begin{equation}
		\calF[f\circ g](x)=\calF[f](g(x)). \footnote{Cartesian product notation and column vector notation are used interchangeably.}
		\end{equation}
		\item Let $g:\R^m\rightarrow\R^{p\times n}$ (i.e. matrix valued) be $C^0$ and $f:\R^m\rightarrow\R^n$ be locally bounded; then
		\begin{equation}
		\calF[gf](x)=g(x)\calF[f](x)
		\end{equation}
		where $gf(x):=g(x)f(x)\in\R^p$.
	\end{enumerate}
\end{property}

A \emph{Filippov solution} of the differential equation $\dot{x}(t)=\funcRdRd(x(t))$ on $[0,T]\subset\R$ is an absolutely continuous function $x:[0,T]\rightarrow\R^n$ that satisfies the differential inclusion
\begin{equation}\label{e:differential_inclusion}
\dot{x}(t)\in \calF[\funcRdRd](x(t))
\end{equation}
for almost all $t\in[0,T]$. A Filippov solution $t\mapsto x(t)$ is \emph{maximal} if it cannot be extended forward in time, that is, if $t\mapsto x(t)$ is not the result of the truncation of another solution with a larger interval of definition. Since the Filippov solutions of a discontinuous system \eqref{e:differential_inclusion} are not necessarily unique, we need to specify two types of invariant set. A set $\calR\subset\R^n$ is called \emph{weakly invariant} for \eqref{e:differential_inclusion} if, for each $x_0\in \calR$, at least one maximal solution of \eqref{e:differential_inclusion} with initial condition $x_0$ is contained in $\calR$. Similarly, $\calR\subset \R^n$ is called \emph{strongly invariant} for \eqref{e:differential_inclusion} if, for each $x_0\in \calR$, every maximal solution of \eqref{e:differential_inclusion} with initial condition $x_0$ is contained in $\calR$. For more details, see \cite{cortes2008,filippov1988}.

Let $\funcRdR$ be a map from $\R^n$ to $\R$. We say that the function $\funcRdR$ is \emph{regular} at $x$ as in \cite{Clarke1990optimization}. In particular, a convex function is regular (see e.g.,\cite{Clarke1990optimization}).


If $V:\R^n\rightarrow\R$ is locally Lipschitz, then its \emph{generalized gradient} $\partial V:\R^n\rightarrow 2^{\R^n}$ is defined by
\begin{equation}
\partial V(x):=\mathrm{co}\Big\{\lim_{i\rightarrow\infty} \nabla
V(x_i):x_i\rightarrow x, x_i\notin S\cup \Omega_{\funcRdR} \Big\},
\end{equation}
where $\nabla$ denotes the gradient operator, $\Omega_{\funcRdR} \subset\R^n$ denotes the set of points where $V$ fails to be differentiable and $S\subset\R^n$ is a set of measure zero that can be
arbitrarily chosen to simplify the computation. Namely, the resulting set $\partial V(x)$ is independent of the choice of $S$ \cite{Clarke1990optimization}.

Given a set-valued map $\calF:\R^n\rightarrow 2^{\R^n}$, the \emph{set-valued Lie derivative} $\calL_{\calF}\funcRdR:\R^n\rightarrow 2^{\R^n}$ of a locally Lipschitz function $V:\R^n\rightarrow \R$  with respect to $\calF$ at $x$ is defined as
\begin{equation}\label{e:set-valuedLie}
\begin{aligned}
\calL_{\calF}V(x) := & \big\{ a\in\R \mid \exists\nu\in\calF(x) \textnormal{ such that } \\
& \quad \zeta^T\nu=a,\, \forall \zeta\in \partial V(x)\big\}.
\end{aligned}
\end{equation}
If $\calF$ takes convex and compact values, then for each $x$, $\calL_{\calF}\funcRdR(x)$ is a closed and bounded interval in $\R$, possibly empty.

The following result is a generalization of LaSalle's invariance principle for discontinuous differential equations \eqref{e:differential_inclusion} with non-smooth Lyapunov functions.
\begin{theorem}[LaSalle Invariance Principle, \cite{Cortes2006}]\label{th:LaSalle_stability}
Let $V:\R^n\rightarrow\R$ be a locally Lipschitz and regular function. Let $S\subset \R^n$ be compact and strongly invariant for \eqref{e:differential_inclusion} and assume that $\max \calL_{\calF[f]} V(x)\leq 0$ for all $x\in S$, where we define $\max\emptyset=-\infty$. Let
\begin{equation}
Z_{\calF[f],V}= \big\{ x\in\R^n \;\big|\; 0\in\calL_{\mathcal{F}[f]}V(x) \big\}.
\end{equation}
Then, all solutions $x:[0,\infty)\rightarrow \R^n$ of \eqref{e:differential_inclusion} with $x(0)\in S$ converge to the largest weakly invariant set $M$ contained in
\begin{equation}
S\cap\overline{Z_{\calF[f],V}}.
\end{equation}
Moreover, if the set $M$ consists of a finite number of points, then the limit of each solution starting in $S$ exists and is an element of $M$.
\end{theorem}
A result on finite-time convergence is stated next, which will form the basis for our results on finite-time consensus.
\begin{proposition}[\cite{Cortes2006}]\label{p:finite_time_convergence}
Under the same assumptions as in Theorem \ref{th:LaSalle_stability}, if $\max \calL_{\calF[\funcRdRd]} f(y)<\varepsilon<0$ a.e.\ on $S\setminus Z_{\calF[f],V}$, then $Z_{\calF[f],V}$ is attained in finite time.
\end{proposition}

\section{Basic model}\label{ss:basic_model}

Consider a multi-agent system composed by $n$ rigid bodies. Denote $\calI=\{1,2,\ldots,n\}$. Suppose the communication network among the agents is an undirected connected graph denoted as $\calG$ with $n$ nodes and $m$ edges.

Let $R_i(t)\in\SO(3)$ be the attitude matrix of rigid body $i$, and the corresponding axis-angle representation $x_i\in\R^3$ is given as 
\begin{equation}\label{e:basic:log}
  \hat{x}_i=\log (R_i).
\end{equation}
The kinematics of $x_i$ is given by
\begin{equation}\label{e:basic_model}
\dot{x}_i=L_{x_i}\omega_i,
\end{equation}
where $\omega_i$ is the angular velocity of rigid body $i$ relative to the initial frame $\mathcal{F}_W$ resolved in body frame $\mathcal{F}_i$, and the transition matrix $L_{x_i}$ is defined as
\begin{equation}
\begin{aligned}
L_{x_i} & = I_3+\frac{\hat{x}_i}{2}+\Bigg( 1-\frac{\sinc(\|x_i\|)}{\sinc^2(\frac{\|x_i\|}{2})} \Bigg) \Big( \frac{\hat{x}_i}{\|x_i\|} \Big)^2 \\
& = \frac{\sinc(\|x_i\|)}{\sinc^2(\frac{\|x_i\|}{2})} I_3+\Bigg(1-\frac{\sinc(\|x_i\|)}{\sinc^2(\frac{\|x_i\|}{2})}\Bigg) \frac{x_i x_i^\top}{\|x_i\|^2}+\frac{\hat{x}_i}{2} \\
& := L^1_{x_i}+\frac{\hat{x}_i}{2},
\end{aligned}
\end{equation}
where $\sinc(\alpha)$ is defined such that $\alpha \sinc(\alpha) = \sin(\alpha)$ and $\sinc(0) = 1$. The proof can be found in \cite{schaub}. We note that for $\|x_i\|\in [0,\pi]$, the function $\frac{\sinc(\|x_i\|)}{\sinc^2(\frac{\|x_i\|}{2})}$ is concave and belongs to $[0,1]$. Then we have the symmetric part of $L_{x_i}$, namely $ L^1_{x_i}$, is positive semidefinite, i.e., for any $z\in\R^3$, $z^\top L_{x_i}z\geq 0$. More precisely, if $\|x_i\|\in [0,\pi)$,  $L^1_{x_i}$ is positive definite. Notice that $L_{x_i}$ is Lipschitz on $B_r(0)$ for any $r<\pi$ (see \cite{Thunberg2014auto}).

System \eqref{e:basic_model} can be written in a compact form as 
\begin{equation}\label{e:plant}
\dot{x}=L_{x}\omega,
\end{equation}
where
\begin{equation}
\begin{aligned}
x & =[x^\top_1,\ldots,x^\top_n]^\top, \\
L_x & =\bdiag{L_{x_1},\ldots, L_{x_n}}, \\
\omega & =[\omega^\top_1, \ldots,\omega^\top_n]^\top.
\end{aligned}
\end{equation}

By defining the consensus space as
\begin{align}
\calC = \big\{ x\in\R^{3n} \mid \exists \bar{x}\in\R^3 \textnormal{ such that } x = \mathds{1}\otimes \bar{x} \big\},
\end{align}
we say the states of the system converge to consensus in finite time if for any initial condition there exists a time $t^*>0$ such that $x=[x_1,\ldots,x_n]^\top$ converge to a \emph{static} vector in $\calC$ as $t\rightarrow t^*$. 

In this paper, we shall design the control input $u_i$ such that the states of the system \eqref{e:plant} converge to consensus in finite time. Our method is motivated by a type of discontinuous protocols which fell into one major category of finite time consensus actuator, see e.g., \cite{Chen2011,Cortes2006,Hui2010}. As a result of introducing discontinuity, we shall understand the trajectories of the final closed-loop in the sense of Filippov. 

%

\section{Controller design: absolute rotation case}\label{s:vector}

In this section, we shall construct one controller which can guarantee the finite time synchronization for the system \eqref{e:plant}. We propose the following discontinuous control protocol
\begin{equation}\label{e:controller2}
\omega_i = \sum_{j\in N_i} \sign(x_j-x_i)
\end{equation}
where the $\sign$ function is taken component-wise and $N_i$ is the set of the neighbors of agent $i$. Notice that the control law only uses coarse information which is in the similar flavor of binary control \cite{Chen2011}.

Now the closed loop is obtained by using \eqref{e:controller2} and \eqref{e:plant}
\begin{equation}\label{e:system_ab}
\dot{x}_i = L_{x_i}\sum_{j\in N_i} \sign(x_j-x_i).
\end{equation}
The stacked version of system \eqref{e:system_ab} can be written as
\begin{equation}\label{e:system_ab_compact}
\dot{x} = - L_x \hat{B} \sign\big(\hat{B}^\T x\big)
\end{equation}
where $B$ is the incidence matrix of the underlying graph and $\hat{B}=B\otimes I_3$.
To handle the discontinuity of the right hand side of \eqref{e:system_ab}, we understand the solution in the sense of Filippov; namely, as solutions of the following differential inclusion:
\begin{equation}\label{e:system_ab_compact_fili}
\begin{aligned}
\dot{x} &\in \calF[ - L_x \hat{B} \sign\big(\hat{B}^\T x \big) ](x) \\
& = - L_x \hat{B} \calF[\sign(\hat{B}^\T x)] (x)\\
& := \calF_1(x),
\end{aligned}
\end{equation}
where the second equality is based on  Property \ref{p:calculus for Filippov} $4)$ and the fact that $L_{x_i}$ is continuous for $\|x_i\|\in [0,\pi)$.
By using the Property \ref{p:calculus for Filippov}, we can enlarge the differential inclusion $\calF_1$ as follows:
\begin{align}
\calF_1(x) & \subset - L_x \hat{B} \bigtimes_{i=1}^{3n} \calF[\sign]((\hat{B}^\top)_i x )\\
& := \calF_2(x),
\end{align}
where $(\hat{B}^\top)_i$ is the $i$th row of $\hat{B}^\T$ and the set-valued function $\calF[\sign]$ is defined as 
\begin{equation}
\calF[\sign](x) = 
\begin{cases}
1 & \textrm{ if } x>0,\\
[-1,1] & \textrm{ if } x=0, \\
-1 & \textrm{ if } x<0.
\end{cases}
\end{equation}

One problem we shall try to avoid for the implementation of control law \eqref{e:controller2} is the singularity of the axis-angular representation of $SO(3)$ at $\pi$. The following lemma provides some strongly invariant sets that will not exhibit this singularity. 

\begin{lemma}\label{l: invariant_set_control2}
The set $S(C) = \{ x\in\R^{3n}| \sum_{i=1}^{n}\|x_i\|_2^2< C \}$ with $C< 4\pi^2$ is strongly invariant for the differential inclusion \eqref{e:system_ab_compact_fili}. Moreover, all the solutions of \eqref{e:system_ab_compact_fili} converge to consensus asymptotically.
\end{lemma}

\begin{proof}
We will employ the Lyapunov functions $V(x)=\frac{1}{2}x^\top x=\frac{1}{2}\sum_{i=1}^{n}x^\top_i x_i$, which is regular, to show the conclusion holds for the bigger inclusion $\calF_2$.
	
Since $V$ is smooth, the set-valued Lie-derivative $\calL_{\calF_2}V(x)$ is given as
	\begin{equation}
	\begin{aligned}
	\calL_{\calF_2} V(x) & = x^\top \calF_2(x) \\
	& = -x^\top \hat{B} \bigtimes_{i=1}^{3n} \calF[\sign]((\hat{B}^\top)_i x ),
	\end{aligned}
	\end{equation}
	where the last equality is implied by the fact that $L_{x_i}$ is well-defined when $\|x_i\|<2\pi$, which is satisfied by the elements in $S(C)$, and $x^\top_iL_{x_i}=x^\top_i$. Furthermore, notice that 
	\begin{equation}
	\begin{aligned}
	& -x^\top \hat{B} \bigtimes_{i=1}^{3n} \calF[\sign]((\hat{B}^\top)_i x ) \\  = & - \sum_{(i,j)\in\calE} (x_i-x_j)^T \bigtimes_{k=1}^{3}\calF[\sign](x_{i_k}-x_{j_k}) \\
	\subset & \R_{\leq 0},
	\end{aligned}
	\end{equation}
	which indicates that the sum of the norm is not increasing along the trajectories when $C<4\pi^2$. Hence the set $S(C)$ is strongly invariant. Notice that the boundedness of the trajectories is also guaranteed. 
	
	Finally, by Theorem \ref{th:LaSalle_stability}, we have that the Filippov solution of system \eqref{e:system_ab_compact_fili} will asymptotically converge to the set
	\begin{equation}
	\Omega  = \overline{\big\{ x\in\R^{3n} \;\big|\; 0\in\calL_{\calF_2}V(x) \big\}},
	\end{equation}
	which is equivalent to the set $\calC$. Then the conclusion follows.
\end{proof}

As we have seen in Lemma \ref{l: invariant_set_control2}, the set $S(C)$ which is defined for the sum of the $\ell_2$ norm of all the states. Hence if $C>\pi^2$, the $\max_i \|x_i\|, i\in\calI$ might be larger than $\pi$ along the evolution. As we shall show in Section \ref{s:simulation}, this is indeed the case. This phenomenon introduces singularity to the axis-angular representation. Therefore, the method we develop in this paper can only apply to the case where the initial condition of \eqref{e:system_ab_compact_fili} belongs to $S(C)$ with $C<\pi^2$. This is equivalent to assume that the initial rotation of the agents is close enough to the origin in $SO(3)$, namely $\sum_{i=1}^nd_R^2(I,R_i(0))<\pi^2$. Now we formulate our main result as follows.

\begin{theorem}\label{th:main}
	Assume that the underlying graph $\calG$ is connected and the initial rotations of the agents satisfy $\sum_{i=1}^nd_R^2(I,R_i(0))<\pi^2$, then the controller \eqref{e:controller2} achieves attitude synchronization in finite time.
\end{theorem}

\begin{proof}
The assumption of the initial rotation is equivalent to $x(0)\in S(C)$ with $C<\pi^2$. By Lemma \ref{l: invariant_set_control2}, we have that $S(C)$ is strongly invariant, which implies that $\|x_i(t)\|<\pi$ for all $i\in\calI$ and $t\geq 0$.


In this proof we use the Lyapunov function $$V = \|\hat{B}^\T x\|_1$$ which is convex, hence regular. By definition, the generalized gradient of $V$ is given as
	\begin{equation}
	\partial V(x)= \{ \zeta \mid \zeta \in \hat{B}\calF[\sign(\hat{B}^\T x)](x) \}.
	\end{equation}
Now the set-valued Lie derivative $\calL_{\calF_1} V(x)$ is given as
	\begin{equation}
	\begin{aligned}
	\calL_{\calF_1} V(x)  = & \{a\in\R \mid \exists \nu\in \calF_1(x) \textnormal{ such that } \\
	 & a = \nu^\top \zeta, \forall \zeta\in\partial V(x) \}. \\
	\end{aligned}
	\end{equation}

Next, let $\Psi$ be defined as
\begin{equation}
\Psi = \big\{ t\geq 0 \mid \textnormal{both } \dot{x}(t) \textnormal{ and } \tfrac{d}{dt}V(x(t)) \textnormal{ exist} \big\}.
\end{equation}
Since $x$ is absolutely continuous (by definition of Filippov solutions) and $V$ is locally Lipschitz, it follows that $\Psi=\R_{\geq 0}\setminus\bar{\Psi}$ for a set $\bar{\Psi}$ of measure zero. Moreover, by Lemma~1 in~\cite{Bacciotti1999}, we have
\begin{equation}
\frac{d}{dt}V(x(t))\in \calL_{\calF_1}V(x(t))
\end{equation}
for all $t\in\Psi$, such that the set $\calL_{\calF_1}V(x(t))$ is nonempty for all $t\in\Psi$. For $t\in\bar{\Psi}$, we have that $\calL_{\calF_1}V(x(t))$ is empty, and hence $\max \calL_{\calF_1}V(x(t)) = -\infty < 0$ by definition. Therefore, we only consider $t\in\Psi$ in the rest of the proof.

Notice that for any $\nu\in\calF_1(x)$, there exists $\tilde{\nu}\in \hat{B}\calF[\sign(\hat{B}^\T x)](x)$ such that
	\begin{equation}
	\nu=-L_x\tilde{\nu}.
	\end{equation}
This implies that $\forall a\in\calL_{\calF_1} V(x)$, there exists $\tilde{\nu}$ such that 
\begin{equation}\label{e:element in LF3}
a = -\tilde{\nu}^\T L_x^\T \zeta, \quad \forall \zeta\in\partial V.
\end{equation}
Since the vector $\tilde{\nu}\in\partial V(x)$, then we have for any $a\in\calL_{\calF_1} V(x)$, it must equal to $-\tilde{\nu}^\T L_x\tilde{\nu}$ for some $\tilde{\nu}\in\partial V(x)$. By the positive definiteness of $L_x$, we have $\calL_{\calF_1} V(x)\subset\R_{\leq 0}.$

We shall show finite time convergence for the case that $x\notin \calC$. Without loss of the generality, we assume that the first coordinations of $x_i, i\in\calI$ are not synchronized. Denote the set $\alpha_1(x)=\{i\in\calI | x_{i_1}=\arg\max_{\ell\in\calI}x_{\ell_1}\}$. Then for any $\tilde{\nu}\in \hat{B}\calF[\sign(\hat{B}^\T x)](x)$, we have 
\begin{equation}
(\sum_{i\in\alpha_1(x)}\tilde{\nu}_i)_1\geq 1.
\end{equation}
Furthermore, for any $a\in \calL_{\calF_1} V(x)$, equation \eqref{e:element in LF3} should hold for all $\zeta \in \hat{B}\calF[\sign(\hat{B}^\T x)](x)$, then it should also hold for $\zeta=\tilde{\nu}$. This implies that
\begin{equation}
\begin{aligned}
a = & - \tilde{\nu}^T L_x \tilde{\nu} \\
 = & -\sum_{i\in\calI} \tilde{\nu}_i^T L^1_{x_i} \tilde{\nu}_i \\ 
 \leq & - \sum_{i\in\calI} \lambda_{\min} \|\tilde{\nu}_i\|_2^2 \\
 \leq & - \sum_{i\in\alpha_1(x)} \lambda_{\min} \|\tilde{\nu}_i\|_2^2 \\
 \leq & - \sum_{i\in\alpha_1(x)} \lambda_{\min} (\tilde{\nu}_{i_1})^2 \\
 \leq & -  \frac{\lambda_{\min}}{|\alpha_1(x)|} (\sum_{i\in\alpha_1(x)}\tilde{\nu}_{i_1})^2 \\
 \leq & -  \frac{\lambda_{\min}}{|\alpha_1(x)|},
\end{aligned}
\end{equation}
where $\lambda_{min}$ is the minimum eigenvalue of all the symmetric part of $L_{x_i}$, i.e., $L^1_{x_i}$. Notice that $\lambda_{min}$ exists by the fact that $L^1_{x_i}>0$ and the function $\frac{\sinc(\|x_i\|)}{\sinc^2(\frac{\|x_i\|}{2})}$ is positive concave for $\|x_i\|\in [0,\pi)$. 
By Proposition \ref{p:finite_time_convergence}, the conclusion follows.
\end{proof}

\begin{remark}
In Theorem \ref{th:main}, the finite-time synchronization is shown for the axis-angular representation.  The kinetic of the attitude matrix of rigid bodies, i.e. $R_i, i\in\calI$, is determined together by the kinetic of $x_i,i\in\calI$ and the map \eqref{e:basic:log}.
\end{remark}

\section{Simulation}\label{s:simulation}

In this section we demonstrate the main result by an example. 
Consider the system \eqref{e:plant} with $x_i\in\R^3$ defined on the graph given as in Fig. \ref{fig:ex_topology}.

\begin{figure}
\centering
\begin{tikzpicture}[scale=1.80]
    \node[shape=circle,draw=black,scale=0.8] (A) at (0,0) {1};
    \node[shape=circle,draw=black,scale=0.8] (B) at (1,0) {2};
    \node[shape=circle,draw=black,scale=0.8] (C) at (1.5,1.0) {3};
    \node[shape=circle,draw=black,scale=0.8] (D) at (2,0) {4};
    \node[shape=circle,draw=black,scale=0.8] (E) at (3,0) {5};

    \path [-](A) edge (B);
    \path [-](B) edge (C);
    \path [-](B) edge (D);
    \path [-](D) edge (C);
    \path [-](D) edge (E);  
\end{tikzpicture}
\caption{The underlying topology for the system in Section \ref{s:simulation}}\label{fig:ex_topology}
\end{figure}
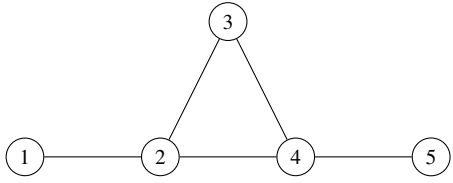

In the first scenario we consider the case that the initial condition $x_i(0), i\in\calI$ belongs to the set $S(C)$ with $C<\pi^2$. The trajectories of system \eqref{e:system_ab_compact_fili} is depicted in Fig.~\ref{fig:ex_FTC}. Here we can see that the closed-loop achieves finite-time consensus. The evolution of the Lyapunov function $V(x)=\frac{1}{2}x^\top x=\frac{1}{2}\sum_{i=1}^{n}x^\top_i x_i$ is shown in Fig.~\ref{fig:ex_Lyapunov_FTC}.

\begin{figure}
\centering
\includegraphics[width=0.94\textwidth]{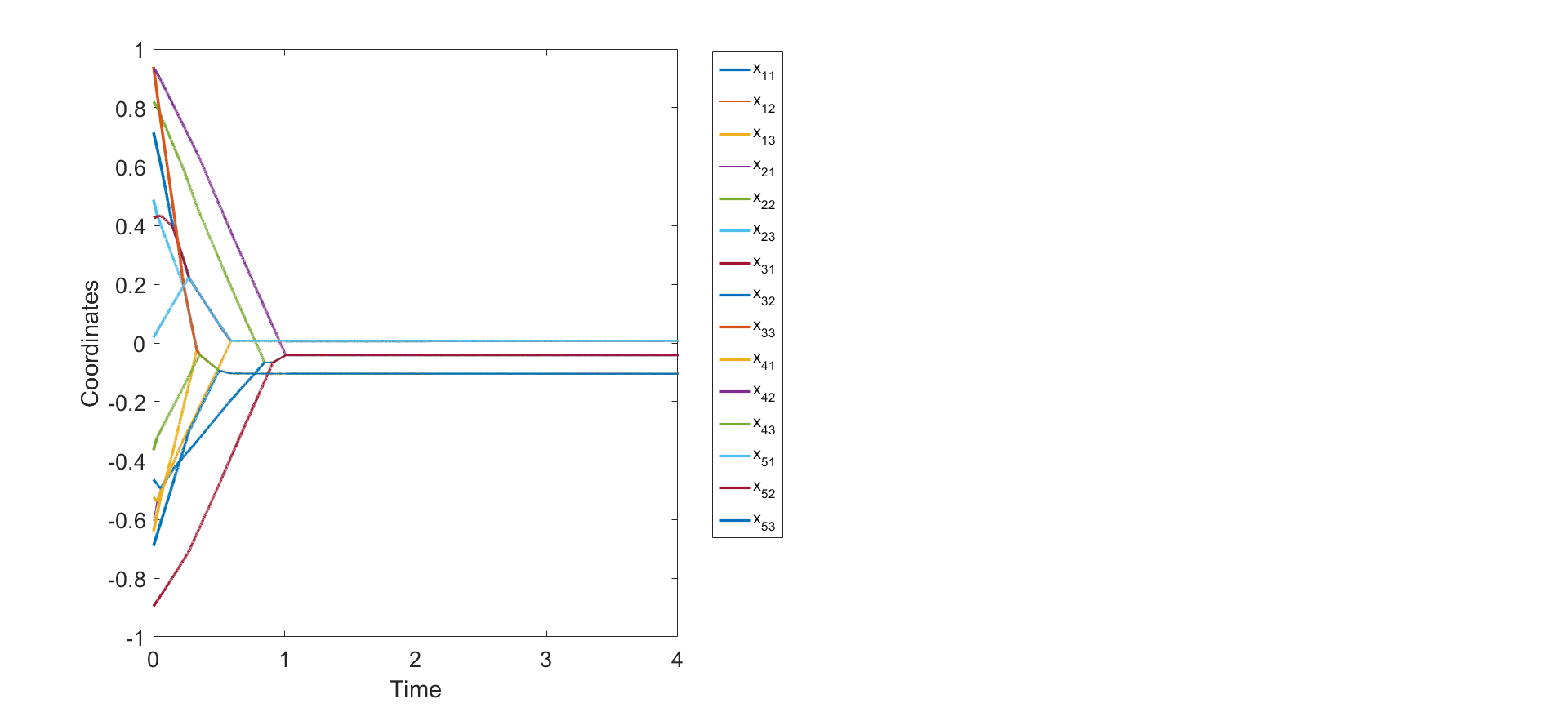}
\caption{The trajectories of system \eqref{e:system_ab_compact_fili} in the first scenario in Section \ref{s:simulation}. Finite-time consensus is achieved.}\label{fig:ex_FTC}
\end{figure}

\begin{figure}
\centering
\includegraphics[width=0.74\textwidth]{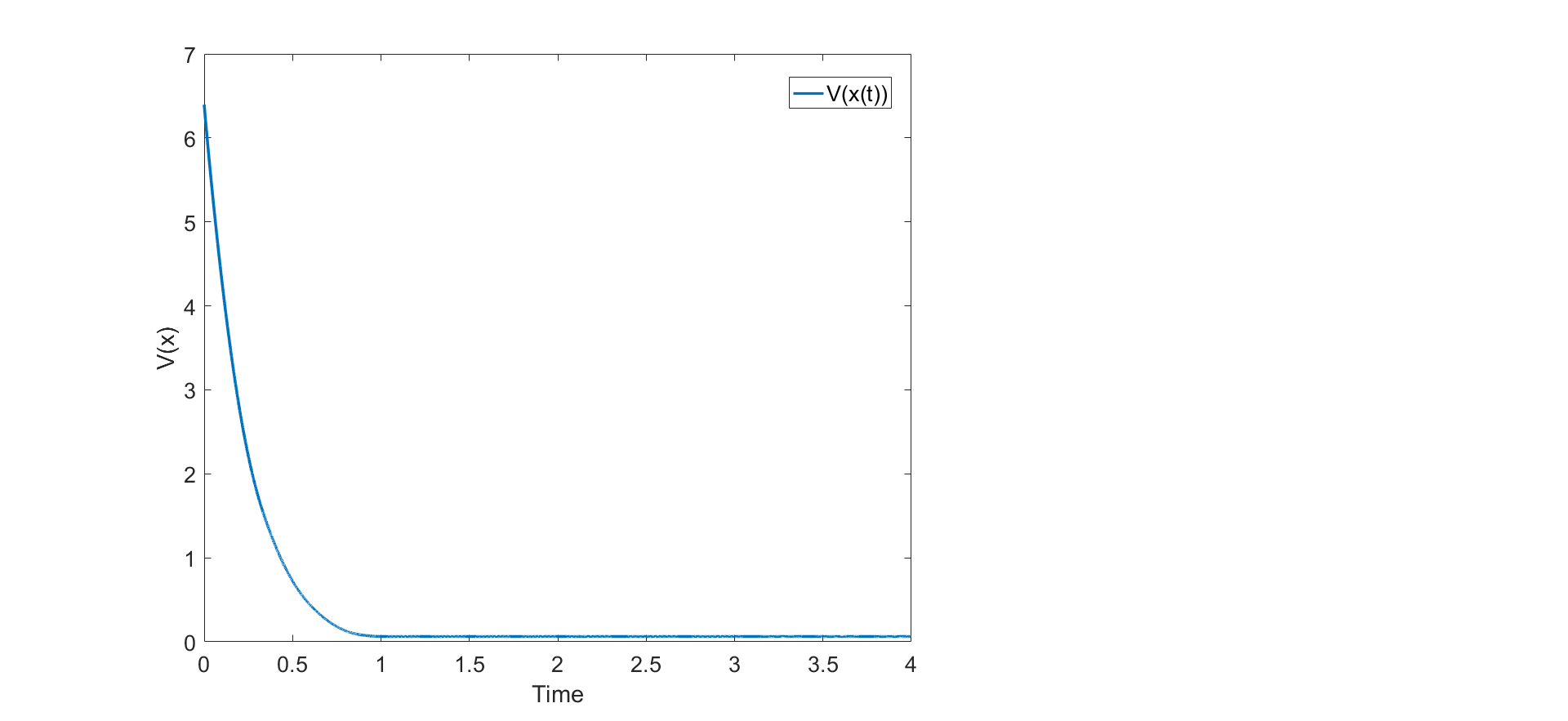}
\caption{The evolution of $V(x)=\frac{1}{2}x^\top x=\frac{1}{2}\sum_{i=1}^{n}x^\top_i x_i$ along the trajectory of system \eqref{e:system_ab_compact_fili}.}\label{fig:ex_Lyapunov_FTC}
\end{figure}

Next, we show that if the initial conditions belong to $S(C)$ but only with $C>\pi^2$, the singularity will exhibit, namely $\max_{i\in\calI} \|x_i\|_2>\pi$. In Fig.~\ref{fig:ex_Lyapunov2_FTC}, we plot the evolution of $\max_{i\in\calI}\|x_i\|_2$ along the trajectories of system \eqref{e:system_ab_compact_fili}. It can be seen that controller \eqref{e:controller2} makes the maximum norm of $x_i$ increase from a number less than $\pi$ to a number larger than it. This introduces the singularity.

\begin{figure}
\centering
\includegraphics[width=0.74\textwidth]{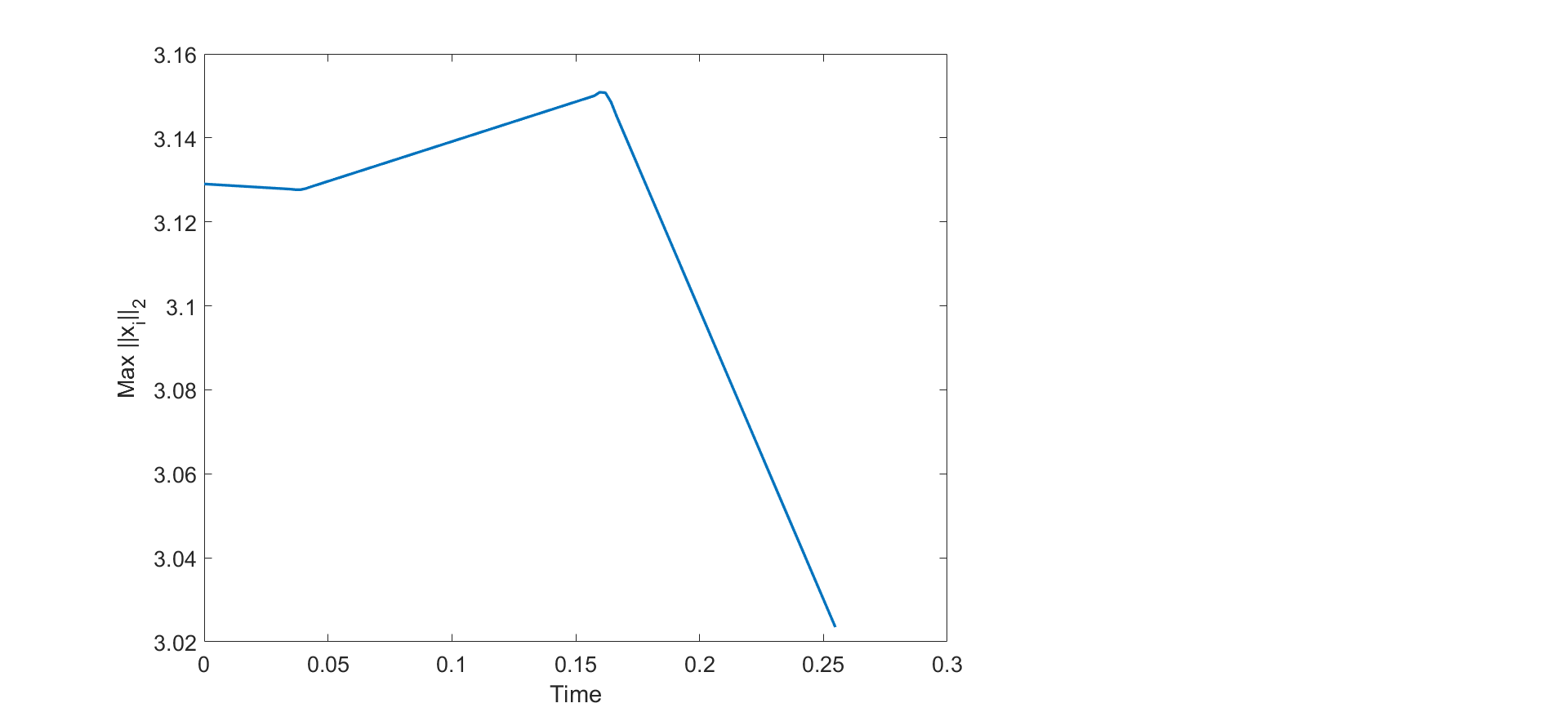}
\caption{The evolution of $V(x)=\frac{1}{2}x^\top x=\frac{1}{2}\sum_{i=1}^{n}x^\top_i x_i$ along the trajectory of the system \eqref{e:system_ab_compact_fili}.}\label{fig:ex_Lyapunov2_FTC}
\end{figure}

\section{Conclusion}

In this paper, we consider the finite-time attitude synchronization problem of a networked rigid bodies system. Motivated by the success of the binary control, we design one distributed discontinuous controller using the signum function. Nonsmooth analysis is employed to prove the stability. However the constraints on the initial condition, namely the initial rotations have to be closed enough to the origin $SO(3)$, limits the application of this controller. Future work will address the more general case, i.e., allowing the initial condition to be arbitrary in $B_I(\pi)$ in $SO(3)$.



\bibliographystyle{plain} 
\bibliography{ref,C:/Users/bartb/Documents/Literature/all}

\begin{thebibliography}{10}

\bibitem{Bacciotti1999}
{A. Bacciotti } and {F. Ceragioli}.
\newblock Stability and stabilization of discontinuous systems and nonsmooth
  {L}yapunov functions.
\newblock {\em ESAIM: Control, Optimisation and Calculus of Variations},
  4:361--376, 1999.

\bibitem{Athanasopoulos2014}
N.~Athanasopoulos, M.~Lazar, C.~B{\"{o}}hm, and F.~Allg{\"{o}}wer.
\newblock {On stability and stabilization of periodic discrete-time systems
  with an application to satellite attitude control}.
\newblock {\em Automatica}, 50(12):3190--3196, 2014.

\bibitem{Bhat00scl}
S.~Bhat and D.~Bernstein.
\newblock A topological obstruction to continuous global stabilization of
  rotational motion and the unwinding phenomenon.
\newblock {\em Systems \& Control Letters}, 39(1):63--70, 2000.

\bibitem{biggs1993algebraic}
N.~Biggs.
\newblock {\em Algebraic Graph Theory}.
\newblock Cambridge Mathematical Library. Cambridge University Press, 1993.

\bibitem{Bollobas98}
B.~Bollobas.
\newblock {\em Modern Graph Theory}, volume 184 of {\em Graduate Texts in
  Mathematics}.
\newblock Springer, New York, 1998.

\bibitem{Bower1964}
J.~Bower and G.~Podraza.
\newblock {Digital implementation of time-optimal attitude control}.
\newblock {\em IEEE Transactions on Automatic Control}, 9(4):590--591, 1964.

\bibitem{Chen2011}
G.~Chen, F.~L. Lewis, and L.~Xie.
\newblock Finite-time distributed consensus via binary control protocols.
\newblock {\em Automatica}, 47(9):1962 -- 1968, 2011.

\bibitem{Clarke1990optimization}
F.~H. Clarke.
\newblock {\em Optimization and Nonsmooth Analysis}.
\newblock Classics in Applied Mathematics. Society for Industrial and Applied
  Mathematics, 1990.

\bibitem{Cortes2006}
J.~Cort\'{e}s.
\newblock Finite-time convergent gradient flows with applications to network
  consensus.
\newblock {\em Automatica}, 42(11):1993--2000, 2006.

\bibitem{cortes2008}
J.~Cort\'{e}s.
\newblock Discontinuous dynamical systems.
\newblock {\em Control Systems, IEEE}, 28(3):36--73, 2008.

\bibitem{Du2011}
H.~Du, S.~Li, and C.~Qian.
\newblock Finite-time attitude tracking control of spacecraft with application
  to attitude synchronization.
\newblock {\em IEEE Transactions on Automatic Control}, 56(11):2711--2717,
  2011.

\bibitem{filippov1988}
A.F. Filippov and F.M. Arscott.
\newblock {\em Differential Equations with Discontinuous Righthand Sides:
  Control Systems}.
\newblock Mathematics and its Applications. Springer, 1988.

\bibitem{Hui2010}
Q.~Hui, W.~M. Haddad, and S.~P. Bhat.
\newblock Finite-time semistability, filippov systems, and consensus protocols
  for nonlinear dynamical networks with switching topologies.
\newblock {\em Nonlinear Analysis: Hybrid Systems}, 4(3):557 -- 573, 2010.

\bibitem{Kowalik1970}
H.~Kowalik.
\newblock {A spin and attitude control system for the Isis-I and Isis-B
  satellites}.
\newblock {\em Automatica}, 6(5):673--682, 1970.

\bibitem{lee2012relative}
Taeyoung Lee.
\newblock Relative attitude control of two spacecraft on so (3) using
  line-of-sight observations.
\newblock In {\em 2012 American Control Conference (ACC)}, pages 167--172.
  IEEE, 2012.

\bibitem{LiuLam2016}
X.~Liu, J.~Lam, W.~Yu, and G.~Chen.
\newblock Finite-time consensus of multiagent systems with a switching
  protocol.
\newblock {\em IEEE Transactions on Neural Networks and Learning Systems},
  27(4):853--862, 2016.

\bibitem{ma2012invitation}
Y.~Ma, S.~Soatto, J.~Kosecka, and S.~Sastry.
\newblock {\em An invitation to 3-d vision: from images to geometric models},
  volume~26.
\newblock Springer Science \& Business Media, 2012.

\bibitem{ZXLi}
R.~Murray, Z.~Li, S.~Sastry, and S.~Sastry.
\newblock {\em A mathematical introduction to robotic manipulation}.
\newblock CRC press, 1994.

\bibitem{paden1987}
B.~Paden and S.~Sastry.
\newblock A calculus for computing filippov's differential inclusion with
  application to the variable structure control of robot manipulators.
\newblock {\em IEEE Transactions on Circuits and Systems}, 34(1):73--82, 1987.

\bibitem{pettersen1996position}
K.~Y. Pettersen and O.~Egeland.
\newblock Position and attitude control of an underactuated autonomous
  underwater vehicle.
\newblock In {\em Proceedings of the 35th IEEE Conference on Decision and
  Control}, volume~1, pages 987--991, 1996.

\bibitem{ren2010distributed}
W.~Ren.
\newblock Distributed cooperative attitude synchronization and tracking for
  multiple rigid bodies.
\newblock {\em IEEE Transactions on Control Systems Technology},
  18(2):383--392, 2010.

\bibitem{schaub}
H.~Schaub and J.~L. Junkins.
\newblock {\em Analytical Mechanics of Space Systems}.
\newblock {AIAA} Education Series, Reston, VA, 2003.

\bibitem{WJ15ac}
W.~Song, J.~Markdahl, X.~Hu, and Y.~Hong.
\newblock Distributed control for intrinsic reduced attitude formation with
  ring inter-agent graph.
\newblock In {\em 54th IEEE Conference on Decision and Control}, pages
  5599--5604. IEEE, 2015.

\bibitem{Thunberg2014auto}
Johan Thunberg, Wenjun Song, Eduardo Montijano, Yiguang Hong, and Xiaoming Hu.
\newblock Distributed attitude synchronization control of multi-agent systems
  with switching topologies.
\newblock {\em Automatica}, 50(3):832 -- 840, 2014.

\bibitem{Tsiotras1994}
P.~Tsiotras and J.~M. Longuski.
\newblock {Spin-axis stabilization of symmetric spacecraft with two control
  torques}.
\newblock {\em Systems {\&} Control Letters}, 23(6):395--402, 1994.

\bibitem{Zong2016}
Q.~Zong and S.~Shao.
\newblock Decentralized finite-time attitude synchronization for multiple rigid
  spacecraft via a novel disturbance observer.
\newblock {\em {ISA} Transactions}, 65:150 -- 163, 2016.

\end{thebibliography}

\end{document}